\documentclass[11pt]{amsart}
\usepackage[utf8]{inputenc}
\usepackage{amsmath}
\usepackage{amsfonts}
\usepackage{amssymb}
\usepackage{tikz-cd}
\usepackage{enumitem}
\usepackage{stmaryrd}
\usetikzlibrary{decorations.markings}
\usetikzlibrary{patterns,patterns.meta}
\tikzset{double line with arrow/.style args={#1,#2}{decorate,decoration={markings,%
mark=at position 0 with {\coordinate (ta-base-1) at (0,1pt);
\coordinate (ta-base-2) at (0,-2pt);},
mark=at position 1 with {\draw[#1] (ta-base-1) -- (0,1pt);
\draw[#2] (ta-base-2) -- (0,-2pt);
}}}}
\tikzset{Equal/.style={-,double line with arrow={-,-}}}

\newcommand\N{\mathbb{N}}
\newcommand\C{\mathbb{C}}

\newcommand{\LC}{\mathcal{L}}

\newcommand{\Mov}{\mathrm{Mov}\,}

\newcommand{\Vect}{\mathrm{Vect}_k}

\newtheorem{theorem}{Theorem}
\newtheorem*{theorem*}{Theorem}

\newtheorem{corollary}[theorem]{Corollary}
\newtheorem{lemma}[theorem]{Lemma}

\theoremstyle{definition}
\newtheorem{definition}{Definition}

\theoremstyle{remark}

\let\<\langle
\let\>\rangle

\let\uml\"
\thanks{}
\title{Isomorphic Loday functors of non-homeomorphic spaces} 

\author{Igor Baskov}

\begin{document}
\newcommand{\Addresses}{{
  \bigskip
  \footnotesize
\noindent
  \textsc{St. Petersburg Department of Steklov Mathematical Institute\\of Russian Academy of Sciences}\par\noindent
  \textit{E-mail address}: \texttt{baskovigor@pdmi.ras.ru}
}}

\maketitle
\begin{abstract}
Each commutative algebra $A$ gives rise to a representation $\mathcal{L}_A$, which we call the Loday functor of $A$, of the category $\Omega$ of finite sets and surjective maps.
In this paper we present two (infinite-dimensional) non-isomorphic algebras over $\mathbb{C}$ with isomorphic Loday functors -- the algebras of continuous functions on the M\"obius strip and on the cylinder.
\end{abstract}

\section{Introduction.}

Let $\Omega$ be the category whose objects are the sets $\<n\> = \lbrace 1,\dots,n\rbrace$, $n\in\N$, and whose morphisms are surjective maps.
A non-unital commutative $k$-algebra $A$ over a field $k$ gives rise to a functor $\LC_A:\Omega \to \Vect$ that takes an object $\<n\>$ to the vector space $A^{\otimes n}$ and a surjection $\sigma:\<n\>\to \<m\>$ to the linear map $\sigma^*:A^{\otimes n}\to A^{\otimes m}$,
$$a_1\otimes\dots\otimes a_n \mapsto b_1\otimes\dots\otimes b_m,$$
$$b_j = \prod_{i\in \sigma^{-1}(j)}a_i.$$
Similar functors were considered in \cite{loday}.

In~\cite{podkorytov1} Podkorytov proves that if for finite-dimensional algebras $A$ and $B$ over an algebraically closed field $k$ the functors $\LC_A$ and $\LC_B$ are isomorphic, then the algebras $A$ and $B$ are isomorphic.

In this paper we present two non-isomorphic algebras over $\C$ with isomorphic Loday functors.
In particular, we show that the $\C$-algebra of continuous functions on the cylinder and the $\C$-algebra of continuous functions on the M\"{o}bius strip have isomorphic Loday functors.
A related result in the discrete case was obtained in~\cite{podkorytov2}.

It is unclear what information about the space $X$ can be recovered from the functor $\LC_{C(X)}$.
By Gelfand--Kolmogorov theorem if for compact Hausdorff spaces $X$ and $Y$ the algebras $C(X)$ and $C(Y)$ are isomorphic as algebras then $X$ and $Y$ are homeomorphic.
As a result of Theorem~\ref{thm:main} we see that from an isomorphism of $\LC_{C(X)}$ and $\LC_{C(Y)}$ one cannot conclude that the spaces $X$ and $Y$ are homeomorphic.

\subsection*{Acknowledgements}
The proof in this paper is based on ideas from~\cite{podkorytov2}.

\section{Projections and twists.}

Let $X$ and $Y$ be topological spaces.
Denote by $\<Y^X\>$ the linear span over $\C$ of the set of all continuous maps $X\to Y$.
We call an element of $\<Y^X\>$ an \emph{ensemble}.
An ensemble
$$A = \sum_{i=1}^n u_ia_i \in \<Y^X\>,$$
where $a_i:X\to Y$ are continuous maps and $u_i\in\C$, gives rise to a linear map $C(Y)\to C(X)$.
Moreover, for each $r\in \N$ such an $A$ induces a linear map
$$A^{(r)}:C(Y)^{\otimes r}\to C(X)^{\otimes r}$$
by the rule
$$A^{(r)} (f_1\otimes\dots\otimes f_r) = \sum_{i=1}^n u_i (f_1\circ a_i)\otimes\dots\otimes (f_r\circ a_i).$$

We say that a topological space $X$ is completely Hausdorff if for any two points $x\neq y$ in $X$ there exists a function $f:X\to [0,1]$ such that $f(x) = 0$ and $f(y) = 1$.

\begin{definition}
Let $r\in\N$. An ensemble $A\in\<Y^X\>$ is called $r$-coherent if $A^{(r)}=0$.
\end{definition}

There is a simple way to check if $A$ is $r$-coherent.
For any finite set $T\hookrightarrow X$ there is the restriction map
$$\<Y^X\>\to \<Y^T\>, \qquad A \mapsto A|_T = \sum_{i=1}^n u_i(a_i|_T).$$

\begin{lemma}\label{lem:FoldabilityCriteria}
Let $X$ and $Y$ be topological spaces and $A\in \<Y^X\>$ be an ensemble.
If $A|_T = 0$ for all finite sets $T\subset X$ of at most $r$ points, then the ensemble $A$ is $r$-coherent.
If $Y$ is completely Hausdorff, then the converse is true.
\end{lemma}

\begin{proof}
First, assume $A|_T = 0$ for all finite sets $T\subset X$ of at most $r$ points.
Take $T = \lbrace x_1,\dots, x_r \rbrace\subset X$ (some of the points may coincide).
In the ensemble $A|_T$ regroup the indices to get
$$A|_T = \sum _{\psi:T\to Y} v_\psi\cdot\psi,$$
where the sum is taken over all maps $T\to Y$.
Then the condition $A|_T = 0$ means that $v_\psi = 0$ for all functions $\psi:T\to Y$.
Take $f_1,\dots, f_r \in C(Y)$ and a point $(x_1,\dots,x_r)\in X^r$.
With the natural inclusion $C(X)^{\otimes r}\hookrightarrow C(X^r)$ we can consider the element
$A^{(r)} (f_1\otimes\dots\otimes f_r)$ as a function on $X^r$.
We have
$$A^{(r)} (f_1\otimes\dots\otimes f_r)(x_1,\dots,x_r) = \sum_{i=1}^n u_i f_1(a_i(x_1))\dots f_r(a_i(x_r)).$$
Again, regroup the summands to get
$$A^{(r)} (f_1\otimes\dots\otimes f_r)(x_1,\dots,x_r) = \sum _{\psi:T\to Y} v_\psi\cdot f_1(\psi(x_1))\dots f_r(\psi(x_r)) = 0.$$

For the other direction, assume that $A^{(r)} (f_1\otimes\dots\otimes f_r) = 0$ for all $f_1,\dots, f_r \in C(Y)$.
Suppose there exists a finite set $T = \lbrace x_1,\dots,x_r\rbrace\subset X$ with $A|_T\neq 0$.
Then there is a function $\varphi:T\to Y$ such that
$$\sum_{i:\,a_i|_T = \varphi} u_i = v_\varphi \neq 0.$$
Define the sets $S_j = \lbrace a_i(x_j)\mid i=1,\dots,n\rbrace$.
Because $Y$ is completely Hausdorff we can choose the continuous functions $f_j:Y\to\C$, $j=1,\dots,r$, such that $f_j(\varphi(x_j)) = 1$ and $f_j = 0$ on $S_j\setminus \lbrace \varphi(x_j) \rbrace$.
Then
$$A^{(r)} (f_1\otimes\dots\otimes f_r)(x_1,\dots,x_r) =$$
$$\sum _{\substack{\psi:T\to Y, \\ \psi\neq\varphi}} v_\psi\cdot f_1(\psi(x_1))\dots f_r(\psi(x_r)) + v_\varphi\cdot f_1(\varphi(x_1))\dots f_r(\varphi(x_r)) = v_\phi\neq 0.$$
\end{proof}

\begin{definition}
Let $X$ and $Y$ be topological spaces.
The following set of functions is called a set of \emph{projections and twists}:
\begin{enumerate}
\item
$p_n:X\to X$, $q_n:Y\to Y$, $n\in\N$, satisfying the identities 
\begin{itemize}
\item
$p_n ^2 = p_n$,
\item
$p_i\circ p_j = p_j\circ p_i$,
\item
$q_n^2 = q_n$,
\item
$q_i\circ q_j = q_j\circ q_i$,
\item
the sets $\Mov p_n = \lbrace x\in X|\,p_n(x)\neq x \rbrace$ are pairwise disjoint,
\item
the sets $\Mov q_n = \lbrace y\in Y|\,q_n(y)\neq y \rbrace$ are pairwise disjoint.
\end{itemize}
\item
$h_n:X\to Y$, $\bar{h}_n:Y\to X$, $n\in\N$, satisfying the identities
\begin{itemize}
\item
$q_j\circ h_i = h_i \circ p_j$,
\item
$p_j\circ \bar{h}_i = \bar{h}_i \circ q_j$,
\item
$\bar{h}_j\circ h_i = p_j \circ p_i$,
\item
$h_i\circ \bar{h}_j = q_i \circ q_j$.
\end{itemize}
\end{enumerate}
If such a collection of maps exists, we say that $X$ and $Y$ are connected by projections and twists.
\end{definition}
In Theorem~\ref{thm:main} we show that if $X$ and $Y$ are connected by projections and twists then the functors $\LC_{C(X)}$ and $\LC_{C(Y)}$ are isomorphic.

Define the ensembles
$$Z_n = \prod_{i=1}^n (1-p_i) \in \<X^X\>,\qquad \bar{Z}_n = \prod_{i=1}^n (1-q_i) \in \<Y^Y\>,$$
where the product is the composition and $1$ is the identity map.
We also set $Z_0 = 1$ and $\bar{Z}_0 = 1$.

As $p_n$ and $q_n$ are idempotent, the elements $1-p_n$ and $1-q_n$ are also idempotent and hence $Z_n$ and $\bar{Z}_n$ are idempotent.

\begin{lemma}\label{lem:ZIdentity}
We have the identities
$$Z_0p_1+Z_1p_2+\dots+Z_{n-1}p_n = 1-Z_n,$$
$$\bar{Z}_0q_1+\bar{Z}_1q_2+\dots+\bar{Z}_{n-1}q_n = 1-\bar{Z}_n$$
for $n\in\N$.
\end{lemma}
\begin{proof}
We prove the first identity by induction on $n\geq 1$.
The base case $n=1$ is trivial.
Assume the identity is true for $n-1$.
We have
$$Z_0p_1+Z_1p_2+\dots+Z_{n-2}p_{n-1}+Z_{n-1}p_n = $$
$$ = 1-Z_{n-1} + Z_{n-1}p_n = 1 -Z_{n-1}(1-p_n) = 1-Z_n.$$
The second identity is analogous.
\end{proof}

Next, define the ensembles
$$T_n = \sum _{i=1}^n \bar{Z}_{i-1}h_i\in \<Y^X\>,\qquad \bar{T}_n = \sum _{i=1}^n Z_{i-1}\bar{h}_i\in \<X^Y\>.$$

\begin{lemma}\label{lem:TCalculation}
For any $n\in\N$ we have the following identities
$$T_n\bar{T}_n = 1-\bar{Z}_n,\qquad \bar{T}_nT_n = 1-Z_n.$$
\end{lemma}
\begin{proof}
We have
$$T_n\bar{T}_n = \left( \sum _{i=1}^n \bar{Z}_{i-1}h_i \right) \left( \sum _{i=1}^n Z_{i-1}\bar{h}_i \right) = \sum _{i,j=1}^n \bar{Z}_{i-1}h_iZ_{j-1}\bar{h}_j.$$
Using the properties of projections and twists we can swap $h_iZ_{j-1}$ to $\bar{Z}_{j-1}h_i$ and simplify to get
$$T_n\bar{T}_n = \sum _{i,j=1}^n \bar{Z}_{i-1}\bar{Z}_{j-1}h_i\bar{h}_j = \sum _{i,j=1}^n \bar{Z}_{i-1}\bar{Z}_{j-1}q_iq_j.$$
Split this sum into the diagonal and the off-diagonal sums:
$$T_n\bar{T}_n = \sum _{i=1}^n \bar{Z}_{i-1}\bar{Z}_{i-1}q_iq_i + \sum _{i,j=1,\;\;i\neq j}^n \bar{Z}_{i-1}\bar{Z}_{j-1}q_iq_j.$$
For the diagonal sum we get
$$\sum _{i=1}^n \bar{Z}_{i-1}\bar{Z}_{i-1}q_iq_i = \sum _{i=1}^n \bar{Z}_{i-1}q_i = 1-\bar{Z}_n$$
by Lemma~\ref{lem:ZIdentity}.
To deal with the off-diagonal sum, consider one summand
$$\bar{Z}_{i-1}\bar{Z}_{j-1}q_iq_j,$$
where $i>j$ (the case $i<j$ is similar).
Note that the element $\bar{Z}_{i-1}$ contains the factor $1-q_j$.
Now, because all of the maps $q_i$ commute with each other, we have $\bar{Z}_{i-1}q_j = 0$ and the whole off-diagonal sum is zero.
The second identity is analogous.
\end{proof}

\begin{lemma}\label{lem:ZAreFoldable}
The ensembles $Z_{n+1}$ and $\bar{Z}_{n+1}$ are $n$-coherent.
\end{lemma}
\begin{proof}
We prove that $Z_{n+1}$ is $n$-coherent.
By Lemma~\ref{lem:FoldabilityCriteria} it is enough to prove that for any finite set $T\subset X$ consisting of at most $n$ points the restriction $Z_{n+1}|_T$ is zero.
As the sets $\Mov p_i$ are pairwise disjoint by definition, there is $j$ such that $\Mov p_j \cap T = \emptyset$.
Hence, $(1-p_j)|_T = 0$ and thus $Z_{n+1}|_T = 0$.
\end{proof}

\begin{corollary}\label{cor:TInduceIdentity}
The ensemble $\bar{T}_{n+1}T_{n+1}$ induces the identity map on $C(X)^{\otimes n}$, and the ensemble $T_{n+1}\bar{T}_{n+1}$ induces the identity map on $C(Y)^{\otimes n}$.
\end{corollary}
\begin{proof}
Follows from Lemma~\ref{lem:TCalculation} and Lemma~\ref{lem:ZAreFoldable}.
\end{proof}

\begin{lemma}\label{lem:TAreCompatible}
For $r\in\N$ we have
$$T_{r+1}^{(r-1)} = T_r^{(r-1)}.$$
\end{lemma}
\begin{proof}
Consider the ensemble $T_{r+1} - T_r = \bar{Z}_rh_{r+1}$.
We have
$$T_{r+1}^{(r-1)} - T_r^{(r-1)} = \bar{Z}_r^{(r-1)}h_{r+1}^{(r-1)} = 0$$
as $\bar{Z}_r^{(r-1)} = 0$ by Lemma~\ref{lem:ZAreFoldable}.
\end{proof}

Take any surjection $\sigma:\<r\>\to \<r-1\>$.
It is clear that for an arbitrary ensemble $A\in\<Y^X\>$ the following diagram commutes
\begin{equation}\label{eq:1}
\begin{tikzcd}
C(Y)^{\otimes r}\arrow[r,"\sigma^*"]\arrow[d,"A^{(r)}"] & C(Y)^{\otimes r-1}\arrow[d,"A^{(r-1)}"]\\
C(X)^{\otimes r}\arrow[r,"\sigma^*"] & C(X)^{\otimes r-1}.
\end{tikzcd}
\end{equation}

\begin{lemma}\label{lem:TCommute}
The following square is commutative
\begin{equation*}
\begin{tikzcd}
C(Y)^{\otimes r}\arrow[r,"\sigma^*"]\arrow[d,"T_{r+1}^{(r)}"] & C(Y)^{\otimes r-1}\arrow[d,"T_{r}^{(r-1)}"]\\
C(X)^{\otimes r}\arrow[r,"\sigma^*"] & C(X)^{\otimes r-1}.
\end{tikzcd}
\end{equation*}
\end{lemma}
\begin{proof}
By Lemma~\ref{lem:TAreCompatible} it is enough to prove that the square
\begin{equation*}
\begin{tikzcd}
C(Y)^{\otimes r}\arrow[r,"\sigma^*"]\arrow[d,"T_{r+1}^{(r)}"] & C(Y)^{\otimes r-1}\arrow[d,"T_{r+1}^{(r-1)}"]\\
C(X)^{\otimes r}\arrow[r,"\sigma^*"] & C(X)^{\otimes r-1}
\end{tikzcd}
\end{equation*}
is commutative, which follows from the general case~(\ref{eq:1}).
\end{proof}

\begin{theorem}\label{thm:main}
Let $X$ and $Y$ be topological spaces that are connected by projections and twists.
Then the Loday functors $\LC_{C(X)}$ and $\LC_{C(Y)}$ are isomorphic.
\end{theorem}

\begin{proof}
Define the morphism $\Phi:\LC_{C(X)}\to\LC_{C(Y)}$ with the $r$-th component $\Phi:C(X)^{\otimes r}\to C(Y)^{\otimes r}$ being $\bar{T}_{r+1}^{(r)}$ and $\Psi:\LC_{C(Y)}\to\LC_{C(X)}$ with the $r$-th component $\Psi:C(Y)^{\otimes r}\to C(X)^{\otimes r}$ being $T_{r+1}^{(r)}$.
We have a diagram
\begin{equation*}
\begin{tikzcd}[row sep=3em,column sep=4em]
C(X)\arrow[d, shift right, "\Phi"'] & C(X)^{\otimes 2}\arrow[d, shift right, "\Phi"']\arrow[l,shift left]\arrow[l,shift right] & C(X)^{\otimes 3}\arrow[d, shift right, "\Phi"']\arrow[l,shift left]\arrow[l,shift right] & \dots \arrow[l,shift left]\arrow[l,shift right]\\
C(Y)\arrow[u, shift right, "\Psi"'] & C(Y)^{\otimes 2}\arrow[u, shift right, "\Psi"']\arrow[l,shift left]\arrow[l,shift right] & C(Y)^{\otimes 3}\arrow[u, shift right, "\Psi"']\arrow[l,shift left]\arrow[l,shift right] & \dots \arrow[l,shift left]\arrow[l,shift right]
\end{tikzcd}
\end{equation*}
where each square for a chosen surjection $\sigma:\<r\>\to\<r-1\>$ is commutative.
By Corollary~\ref{cor:TInduceIdentity} the vertical maps are isomorphisms.
Hence $\Phi$ and $\Psi$ are mutually inverse natural isomorphisms of the Loday functors.
\end{proof}

\section{The cylinder and the M\"{o}bius strip.}

In this section we show that the cylinder and the M\"{o}bius strip are connected by projections and twists.
By Theorem~\ref{thm:main} these two non-homeomorphic spaces have isomorphic Loday functors.

Define the cylinder $X$ as the quotient of $[0,1]\times [0,1]$ obtained by identifying the vertical sides via $(0,y)\sim (1,y)$ and the M\"{o}bius strip $Y$ as the quotient of $[0,1]\times [0,1]$ obtained by identifying the vertical sides via $(0,y)\sim (1,1-y)$.

Set $x_n = \dfrac{1}{n+1}$; choose $\epsilon_n>0$ so that the intervals $(x_n-\epsilon_n,x_n+\epsilon_n)$ and $(x_m-\epsilon_m,x_m+\epsilon_m)$ do not intersect for any $n\neq m$.
Looking at the picture below denote by $R_n\subset [0,1]\times [0,1]$ the space on the right and denote by $r_n:[0,1]\times [0,1]\to R_n$ some vertical retraction, symmetric about the line $y=\frac{1}{2}$.

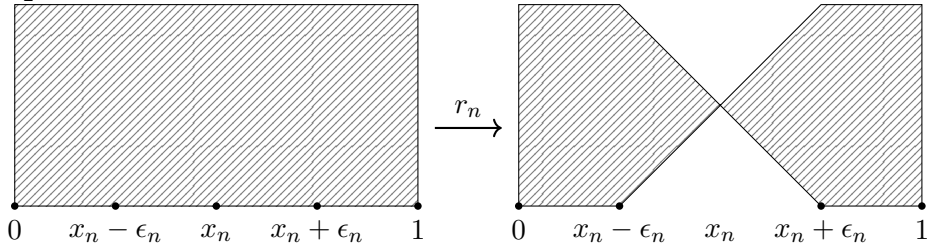
\begin{tikzpicture}[scale=0.67]

\begin{scope}[local bounding box=first]
  \begin{scope}[scale=0.5]
    \fill[pattern=north east lines,pattern color=gray]
      (0,0) rectangle (16,8);
  \end{scope}

  \draw (0,0) -- (2,0) -- (4,0) -- (6,0) -- (8,0)
        -- (8,4) -- (6,4) -- (4,4) -- (2,4) -- (0,4) -- cycle;

  \draw (0,0) node[circle,fill,inner sep=1pt,label=below:$0$]{} ;
  \draw (2,0) node[circle,fill,inner sep=1pt,label=below:$x_n-\epsilon _n$]{} ;
  \draw (4,0) node[circle,fill,inner sep=1pt]{} node[below=3pt]{$x_n$} ;
  \draw (6,0) node[circle,fill,inner sep=1pt,label=below:$x_n+\epsilon _n$]{} ;
  \draw (8,0) node[circle,fill,inner sep=1pt,label=below:$1$]{} ;
\end{scope}

\begin{scope}[xshift=10cm, local bounding box=second]
  \begin{scope}[scale=0.5]
    \fill[pattern=north east lines,pattern color=gray]
      (0,0) -- (4,0) -- (8,4) -- (12,0) -- (16,0)
      -- (16,8) -- (12,8) -- (8,4) -- (4,8) -- (0,8) -- cycle;
  \end{scope}

  \draw (0,0) -- (2,0) -- (4,2) -- (6,0) -- (8,0)
        -- (8,4) -- (6,4) -- (4,2) -- (2,4) -- (0,4) -- cycle;

  \draw (0,0) node[circle,fill,inner sep=1pt,label=below:$0$]{} ;
  \draw (2,0) node[circle,fill,inner sep=1pt,label=below:$x_n-\epsilon _n$]{} ;
  \draw (4,0) node[below=3pt]{$x_n$} ;
  \draw (6,0) node[circle,fill,inner sep=1pt,label=below:$x_n+\epsilon _n$]{} ;
  \draw (8,0) node[circle,fill,inner sep=1pt,label=below:$1$]{} ;
\end{scope}

\draw[->, thick] (first.east) --node[above]{$r_n$} (second.west);

\end{tikzpicture}

The composition
$$[0,1]\times [0,1]\xrightarrow{r_n} R_n \hookrightarrow [0,1]\times [0,1]$$
gives rise to the maps $p_n:X\to X$ and $q_n:Y\to Y$.

Denote by $\gamma_n:R_n\to R_n$ the half-twist of the right half of $R_n$,
$$\gamma_n(x,y) = \begin{cases} (x,y),\;\;0\leq x\leq x_n,\\ (x,1-y),\;\; x_n\leq x\leq 1.\end{cases}$$
The composition
$$[0,1]\times [0,1]\xrightarrow{r_n} R_n\xrightarrow{\gamma_n} R_n \hookrightarrow [0,1]\times [0,1]$$
gives rise to the maps $h_n:X\to Y$ and $\bar{h}_n:Y\to X$.

The reader can easily check that this set of maps defines a set of projections and twists.

\Addresses
\end{document}